\documentclass[11pt]{amsart}
\usepackage{latexsym}
\usepackage{graphicx}
\usepackage{subfigure}
\usepackage{amsmath}
\usepackage{amssymb}
\usepackage{mathrsfs}
\usepackage{epstopdf}
\usepackage[all]{xypic}
\DeclareMathOperator{\Ker}{Ker}

\DeclareMathOperator{\Coker}{Coker}
\DeclareMathOperator{\Sym}{Sym}

\newtheorem{theorem}{Theorem}[section]
\newtheorem{lemma}[theorem]{Lemma}
\newtheorem{prop}[theorem]{Proposition}
\newtheorem{corollary}[theorem]{Corollary}

\newtheorem{conjecture}[theorem]{Conjecture}

\renewcommand{\Im}{\mathop{\rm Im}\nolimits}

\numberwithin{equation}{section}

\title{ on the maximal rank conjecture for line bundles of extremal degree}
\author{Jie Wang}
\thanks{ Department of Mathematics, Ohio State University, Columbus, OH, 43210 (jwang@math.ohio-state.edu).}
\date{\today}
\begin{document}
\maketitle

\begin{abstract} We propose a new method, using deformation theory, to study the maximal rank conjecture. For line bundles of extremal degree, which can be viewed as the first case to test the conjecture, we prove that maximal rank conjecture holds by our new method. 

\end{abstract}

\bigskip

\section{introduction.}

A central problem in curve theory is to describe algebraic curves in a given projective space $\mathbb{P}^r$ ($r\ge3$) with fixed genus and degree. For instance, one wants to describe the ideal of a curve $C\subset\mathbb{P}^r$, and in particular, to know the Hilbert function of $C$, or in geometric terms, how many independent hypersurfaces of each degree $C$ lies on.  A major open problem here is the maximal rank conjecture, that appeared in Eisenbud-Harris \cite{EH2}: 
\begin{conjecture}\label{maximal}
(Maximal rank conjecture) For fixed $d$, $g$, $r\ge3$, let $C$ be a general curve of genus $g$ and $|L|$ be a general $g^r_d$ on $C$, then the multiplication map
\begin{eqnarray}\label{mmap}
\xymatrix{\Sym^kH^0(C,L)\ar[r]^-{\mu^k}&H^0(C,L^{k})}
\end{eqnarray}
is of maximal rank (either injective or surjective) for any $k\ge1$.
 \end{conjecture}
In the case $|L|$ gives an embedding of $C$ into $\mathbb{P}^r$, $\Sym^kH^0(C,L)$ is the space of homogeneous polynomials of degree $k$ in $\mathbb{P}^r$ and $\Ker(\mu^k)$ is just the subspace consisting of those vanishing on $C$. By the Gieseker-Petri theorem, on a general curve $C$, $L^{k}$ is always non-special for $k\ge2$. Thus the dimension of the domain and target of $\mu^k$ are constants only depending on $k$, $d$, $r$ and $g$. Therefore, the maximal rank conjecture (MRC) simply says that the number of independent hypersurfaces containing $C$ is as small as it could be.

Since conjecture \ref{maximal} concerns conditions that are open, it suffices to prove the statement for one point on each component of the parameter space $\mathcal{G}^r_d$ which dominates $\mathcal{M}_g$, where
$$\mathcal{G}^r_d=\{(C,L,V)\ |\ L\text{ line bundle on }C,\ deg\ L=d,\ V\subset H^0(L),\ dimV=r+1\}$$  

However, if genus is large, it is very difficult to write down smooth curves satisfying the conjecture. The classic strategy to deal with this problem is to degenerate smooth curves to some special singular ones and try to verify the conjecture on these singular curves. By degenerating to reducible curves with embedded points and using a rather complex inductive argument, Ballico and Ellia proved the MRC for non-special linear series in \cite{BE}, \cite{BE1}, and \cite{BE2}.  So the main interest now is in the case of special linear series on general curves.

In another development, it is proved in Green and Lazarsfeld \cite{GL1} that any very ample line bundle $L$ on $C$ with 
$$deg\ L\ge 2g+1-2h^1(L)-Cliff(C)$$
or equivalently
$$Cliff(L)<Cliff(C),$$
is projectively normal, where $Cliff(C)$ is the clifford index of $C$:
$$Cliff(C):=\text{min}\{Cliff(A)\ |\ A\ \text{line bundle on}\ C,\ h^0(A)\ge2,\ h^1(A)\ge2\}$$
and
$$Cliff(A)=deg\ A-2r(A).$$
and for a general curve $C$, $Cliff(C)=\lfloor\frac{g-1}{2}\rfloor$. 

It is also showed by Green and Lazarsfeld that the bound $2g+1-2h^1(L)-Cliff(C)$  is the best possible. There are line bundles of degree one less than this bound which are not normally generated. We say a line bundle $L$ on $C$ has extremal degree if 
$$deg\ L=2g-2h^1(L)-Cliff(C),$$
that is,
$$Cliff(L)=Cliff(C).$$

On the other hand, if the maximal rank conjecture were true, we should still expect projective normality for general line bundles of extremal degree on general curves. Thus the extremal degree range should be thought of as the first case to test the maximal rank conjecture. There are four cases according to the value of $h^1(L)$:
\begin{enumerate}
\item[(1)] $h^1(L)=0$. $L$ is non special and the MRC follows from \cite{BE2}.
\vspace{.1cm}
\item[(2)] $h^1(L)=1$. If $g=2l$ even, $L$ is a $g^l_{3l-1}$ and $\rho=l-1$; if $g=2l+1$ odd, $L$ is a $g^l_{3l}$ and $\rho=l$.
\vspace{.1cm}
\item[(3)] $h^1(L)=2$. If $g=2l$ even, $L$ is a $g^{l-1}_{3l-3}$ and $\rho=0$; if $g=2l+1$ odd, $L$ is a $g^{l-1}_{3l-2}$ and $\rho=1$.
\vspace{.1cm}
\item[(4)] $h^1(L)\ge 3$.  The Brill-Noether number is negative. There are no such $g^r_d$'s $(r\ge3)$ on a general curve.
\end{enumerate}

In this paper, we prove the MRC for the remaining open cases $(2)$ and $(3)$.
\begin{theorem}\label{main thm}
Let $C$ be a general curve of genus $g$ ($g\ge 10$ if $g$ even, $g\ge 13$ if $g$ odd), $L$ be a general line bundle of extremal degree on $C$, then $(C,L)$ satisfies the MRC, or equivalently, is projectively normal.
\end{theorem}

By theorem $(4.e.1)$ of (\cite{Gr}), in our degree range, it suffices to prove $\mu^2$ in (\ref{mmap}) is surjective. We apply a new method, using deformation theory, to prove this fact. The general idea of this method is as follows. Instead of looking for some $(C_0, L_0)$ such that  $\mu^2$ is of maximal rank there, consider a one parameter family of pairs $(C_t, L_t)\in\mathcal{W}^r_d$, specializing to some $(C_0, L_0)$ ($C_0$ could be singular) with $\mu^2(0)$ not necessarily of maximal rank. Suppose moreover that all global sections of $L_0$ extends to $L_t$. Then one can construct obstruction maps 
$$\xymatrix{\delta_1: \Ker(\mu^2(0))\ar[r]&\Coker(\mu^2(0))}$$
and inductively 
$$\xymatrix{\delta_{n+1}:\Ker(\delta_n)\ar[r]&\Coker(\delta_n)}$$
such that an element $s\in \Ker(\mu^2(0))$ extends to $\Ker(\mu^2(t))$ modulo $t^{n+1}$ if and only if $\delta_i(s)=0$ for $i=0,...,n.$

For the decreasing sequence
$$\Ker(\mu^2(0))\supset\Ker(\delta_1)\supset...\supset\Ker(\delta_n)\supset...$$
if we can show that the vector space $V=\bigcap_{i}\Ker(\delta_i)$ consisting of elements which deform to $\Ker(\mu^2(t))$ to any order is of \textquotedblleft correct dimension", then $\mu^2(t)$ is of maximal rank. Said differently, it suffices to prove that $\delta_n$ is of maximal rank for some $n\in\mathbb{Z}_+$.

We find a nice singular curve $C_0$ on which the computation  of obstructions is surprisingly simple. Enough information in the obstruction maps $\delta_n$ is captured by the natural multiplication map
$$\xymatrix{\kappa_n:H^0(C_0,L_n)\otimes H^0(C_0,L_{-n})\ar[r]&H^0(C_0,L^2_0)}$$
in the following theorem:
\begin{theorem} \label{new}
  Let $\mathcal{L}\rightarrow\mathcal{C}$ be the total space of a one parameter family $(C_t,L_t)\in\mathcal{W}^{r}_{d}$ degenerating to $(C_0,L_0)$ with $C_0=X\cup Y$ a nodal curve consisting of two smooth curves of genus $g_X$, $g_Y$ meeting at a point $p$ and $\mathcal{C}$ is smooth. Write $L_n=\mathcal{L}(nY)|_{C_0}$.  Suppose all (global) sections of $L_n$ extend to $L_t$ for $|n|\le a$ and the natural map
\begin{eqnarray}\label{kappa}
\xymatrix{\bigoplus_{n=0}^{a}H^0(C_0,L_n)\otimes H^0(C_0,L_{-n})\ar[r]^-{\kappa=\oplus_{n}\kappa_n}&H^0(C_0,L^2_0)}
\end{eqnarray}
is surjective (resp. of rank $=dim_{\mathbb{C}}\Sym^2H^0(L_0)$) for some $a\in\mathbb{Z}_+$, then the multiplication map $\mu^2(t)$ is surjective (resp. injective) for small $t\ne0$. 
\end{theorem}

Notice that $\kappa$ only depends on $(C_0,L_0)$, not on the actual family specializing to it. It seems to the author that such a simple way to describe higher order obstructions is new and should have a lot more applications. 

The significance of theorem \ref{new} is that we are now reduced to finding a smoothable $(C_0,L_0)$ such that  all sections of $L_n$ extend to the nearby fiber and $\oplus_{n=0}^{a}\kappa_n$ (instead of $\kappa_0=\mu^2(0)$) is of required rank. By making a good choice of $(C_0,L_0)$, we manage to prove theorem \ref{main thm} by showing that $\kappa$ in theorem \ref{new} is surjective.  

This paper is organized as follows:

In section $2$, we set up some machinery which measures the obstructions for elements of $\Ker(\mu^k(0))$ to extend to $\Ker(\mu^k(t))$.

In section $3$, we compute the obstruction maps $\delta_n$ for the special degeneration described in theorem \ref{new} and give a proof of this theorem.

Section $4$ contains a proof of the main theorem \ref{main thm}.

Finally, in section $5$, we include some technical facts about canonical bundles on general curves which are needed in the proof of the main theorem.

{\bf Acknowledgements.} The author wishes to thank his advisor Herb Clemens for suggesting the problem and method, valuable discussions and constant support.
  
\section{infinitesimal study of the degeneracy loci}

Let $C_0$ be a reduced l.c.i curve over $\mathbb{C}$ and $L_0$ be a degree $d$ line bundle on $C_0$ with $h^0(L_0)=r+1$. By theorem $4.1$ in \cite{W}, the deformations of the pair $(C_0,L_0)$ are unobstructed. Let $S$ be the versal deformation space of $(C_0,L_0)$, then $S$ is smooth near $(C_0,L_0)$. Let $\mathcal{W}^r_d$ be the subvariety of $S$ consists of $(C,L)$ such that $h^0(L)\ge r+1$.
Consider the multiplication map
\begin{eqnarray}\label{multiplication}
\xymatrix{\Sym^kH^0(C,L)\ar[r]^-{\mu^k}&H^0(C,L^k)}.
\end{eqnarray}

We may think of this map as a morphism between two vector bundles (at least near the point $(C_0,L_0)$) over $\mathcal{W}^r_d$ as $(C,L)$ varies in $\mathcal{W}^r_d$. We are interested in the infinitesimal properties of the locus $D$ consisting of $(C,L)$ such that the multiplication map is not of maximal rank, i.e it is neither injective nor surjective. Our goal is to show that $D$ is a proper subvariety of $\mathcal{W}^r_d$ (assuming $\mathcal{W}^r_d$ irreducible near $(C_0,L_0)$).

Suppose now  that there is a (flat) one parameter family $(C_t,L_t)$ of pairs specializing to $(C_0,L_0)$ such that all sections of $L_0$ extend to $L_t$. If $\mu^k(0)$ is not of maximal rank at $(C_0,L_0)$, then the dimension of $\Ker(\mu^k(0))$ is bigger than expected. We would like to knock down this dimension by showing that only a expected number of independent sections of $\Ker(\mu^k(0))$ can extend to $\Ker(\mu^k(t))$. Thus $\mu^k(t)$ is of maximal rank for $t\ne0$.

Our goal in this section is to set up some machinery which measures the obstructions for elements of $\Ker(\mu^k(0))$ to extend to $\Ker(\mu^k(t))$.

To this end, let $(\mathcal{C},\mathcal{L})$ be the total space of the one parameter family and $(\mathcal{C}_n,\mathcal{L}_n)$ be the restriction of $(\mathcal{C},\mathcal{L})$ to $Spec\frac{\mathbb{C}[t]}{(t^{n+1})}=:Spec R_i$. Let $M_n=\Sym^kH^0(\mathcal{C}_n,\mathcal{L}_n)$, $N_n=H^0(\mathcal{C}_n,\mathcal{L}^k_n)$ and $\mu_n: M_n\rightarrow N_n$ be the multiplication map.

We have $M_{i+1}\otimes_{R_{i+1}}R_i=M_i$, $N_{i+1}\otimes_{R_{i+1}}R_i=N_i$ compatibly with $\mu_i$ for any $i\ge0$. 

\begin{lemma}\label{lemma5.1}
Under the above notations and assumptions, there exist obstruction maps 
$$\delta_{n+1}: \Ker(\delta_n)\rightarrow \Coker(\delta_n)$$
for $n\ge0$ such that $\delta_0=\mu_0$ and $s\in \Ker(\mu_0)$ can be lifted to $\Ker(\mu_n)$ if and only if $\delta_i(s)=0$ for $i=0,...,n.$
\end{lemma}
\begin{proof} 

For each $n\ge0$ consider
\begin{eqnarray}\label{obstruction}
\xymatrix{0\ar[r]&M_0\ar[r]^-{\cdot t^{n+1}}\ar[d]^{\mu_0}&M_{n+1}\ar[r]^-{p_{n+1}}\ar[d]^{\mu_{n+1}}&M_n\ar[r]\ar[d]^{\mu_n}\ar[r]&0\\
0\ar[r]&N_0\ar[r]^-{\cdot t^{n+1}}&N_{n+1}\ar[r]^-{q_{n+1}}&N_n\ar[r]&0}
\end{eqnarray}

Let $\delta_{n+1}':\Ker(\mu_{n})\rightarrow\Coker(\mu_0))$ be the connecting homomorphism of (\ref{obstruction}) from snake lemma. Fix $s_0\in\Ker(\mu_0)$ and $n\ge1$. Suppose $s_0$ has a lifting $s_{n}\in\Ker(\mu_n)$.
Let $s_{n+1}'\in M_{n+1}$ be any lifting of $s_n$. Suppose that $\mu_{n+1}(s_{n+1}')=t^{n+1}v$ with 
\begin{eqnarray}\label{solution}
v=\mu_0(s_0')+\delta_1(p_1(s_1'))+...+\delta_n(p_n(s_n')),
\end{eqnarray}
then 
$$s_{n+1}'-(t^{n+1}s_0'+t^ns_1'+...+ts_n')\in\Ker(\mu_{n+1}).$$
On the other hand, if (\ref{solution}) has no solution for any collection $s_j'\in M_j$ with $p_j(s_j')\in\Ker(\mu_{j-1})$, then $s_0$ has no lifting to $\Ker(\mu_{n+1})$.

Now, simply define $\delta_{n+1}(s_0)=\delta_{n+1}'(s_n)=v$ as an element of 
$$\frac{\Coker(\mu_0)}{\sum_{i=1}^n \Im(\delta_i')}=\Coker(\delta_n).$$
Then $\delta_{n+1}(s_0)$ does not depend on the choice of $s_n$ and is equal to zero in $\Coker(\delta_n)$ if and only if $s_0$ can be lifted to $\Ker(\mu_{n+1})$.

Finally, we check
$$\frac{\Coker(\mu_0)}{\sum_{i=1}^{n+1} \Im(\delta_i')}=\Coker(\delta_{n+1}).$$
\end{proof}

By theorem $3.1$ in \cite{W}, $(\mathcal{C}_1,\mathcal{L}_1)$ determines a tangent vector $\xi\in T_{(C_0,L_0)}S=\mathrm{Ext}^1_{\mathcal{O}_{C_0}}(\mathcal{P}^1_{C_0}(L_0),L_0)$ which annihilates $H^0(C_0,L_0)$. This means exactly that $\xi$ is tangent to $\mathcal{W}^r_d\subset S$ at $(C_0,L_0)$.
$\xi$ is a tangent direction such that the rank of $\mu^k_0$ does not increase if and only if
$$\delta_1: \Ker(\mu^k_0)\longrightarrow \Coker(\mu^k_0)$$
is zero. If this is the case, then every element in $\Ker(\mu^k_0)$ extends to $\Ker(\mu^k_1)$, thus to first order, the rank of the map $\mu^k$ does not increase (It does not decrease either, by lower semicontinuity of the rank). 

On the other hand, if $\delta_1$ is of maximal rank, there are two cases:  

1) $\delta_1$ is injective. No elements of $\Ker(\mu^k_0)$ will extend to $\Ker(\mu^k_1)$, thus for a general $t\ne0$, the multiplication  map (\ref{multiplication}) is injective at $(C_t,L_t)$.

2) $\delta_1$ is  surjective. Only a subspace of $\Ker(\mu^k_0)$ of dimension $dim_{\mathbb{C}}\Ker({\mu^k_0})-dim_{\mathbb{C}}Coker({\mu^k_0})=dim_{\mathbb{C}}\Sym^kH^0(L_0)-h^0(L_0^k)$ will extend to first order, therefore, for the nearby $(C_t,L_t)$, the multiplication map (\ref{multiplication}) is surjective.

Suppose now that $\delta_1$ is not of maximal rank. It is not possible to test if the nearby multiplication map is of maximal rank to first order. We have to look at the higher order obstruction maps $\delta_n$. 

By lemma \ref{lemma5.1}, any $s\in \Ker(\mu^k_0)$ can be extended to $\Ker(\mu^k_n)$ if and only if $\delta_i(s)=0$ for $i=0,...,n$. Let $n$ be the smallest integer such that $\delta_n$ is of maximal rank (if it exists). Since the index of $\delta_i$ $Ind\ \delta_i:=dim_{\mathbb{C}}\Ker(\delta_i)-dim_{\mathbb{C}}\Coker(\delta_i)$ is always constant for any $i$, we see that only a subspace of $\Ker(\mu^k_0)$ of expected dimension (0 if $\delta_n$ injective, $Ind\ \delta_n$ if $\delta_n$ surjective) can be extended to $\Ker(\mu^k_n)$. Therefore, the multiplication maps for nearby fibers are of maximal rank.

 We have proved the following proposition:
\begin{prop}\label{sufficient}
If $\delta_n$ is of maximal rank for some $n\in\mathbb{Z}_+$, the multiplication map $\mu^k$ is of maximal rank for nearby fibers.
\end{prop}

\section{A nice degeneration}

To use proposition \ref{sufficient}, we need to compute the obstruction maps $\delta_n$, which in general is difficult. However, for the $k=2$ case, there is a nice degeneration on which the computation is surprisingly simple.

Let $\mathcal{L}\rightarrow\mathcal{C}$ be the total space of a one parameter family $(C_t,L_t)\in\mathcal{W}^{r}_{d}$ degenerating to $(C_0,L_0)$ with $C_0=X\cup Y$ a nodal curve consisting of two smooth curves of genus $g_X$, $g_Y$ meeting at a point $p$. Write $L_n=\mathcal{L}(nY)|_{C_0}$. Suppose all sections of $L_n$ extend to $L_t$ for $|n|\le a$. Notice that $L_n|_X=L_0|_X(np)$ and $L_n|_Y=L_0|_Y(-np)$, thus $L_n$ only depends on $(C_0,L_0)$, not on the family specializing to it.

The multiplication map
\begin{eqnarray}\label{max}
\xymatrix{\Sym^2H^0(C_0,L_0)\ar[r]^-{\mu(0)}&H^0(C_0,L_0^2)}
\end{eqnarray}
is usually not of maximal rank here, which means dimension of $\Ker(\mu(0))$ is bigger than it should be. 

There are some obvious elements in $\Ker(\mu(0))$. Let $W$ be the subspace of $\Sym^2H^0(L_0)$ spanned by
$$\{\sigma\cdot\tau\ :\ \sigma,\tau\in H^0(L_0),\ \sigma|_X\equiv0,\ \tau|_Y\equiv0,\}.$$
Clearly $W$ is a subspace of $\Ker(\mu(0))$. Let's compute the image of $W$ under
$$\xymatrix{\delta_1: \Ker(\mu(0))\ar[r]&\Coker(\mu(0)).}$$
Let $\tilde\sigma$, $\tilde\tau$ be sections of $\mathcal{L}$ which extend $\sigma$, $\tau$ respectively (They always exist, since by assumption, all sections of $L_0$ extend to $L_t$). Then $\tilde\sigma=\tilde\sigma's_X$, $\tilde\tau=\tilde\tau's_Y$, where $s_X$ (resp. $s_Y$) is a section of $\mathcal{O}_{\mathcal{C}}(X)$ (resp. $\mathcal{O}_{\mathcal{C}}(Y))$ which vanishes exactly on $X$ (resp $Y$) and $\tilde\sigma'$ (resp. $\tilde\tau'$) is a section of $\mathcal{L}(-X)$ (resp. $\mathcal{L}(-Y)$). By the construction of $\delta_1$,
$$\delta_1(\sigma\cdot\tau)=\frac{\tilde\sigma\tilde\tau}{t}|_{C_0}=\frac{\tilde\sigma's_X\tilde\tau's_Y}{t}|_{C_0}=\tilde\sigma'\tilde\tau'|_{C_0}\ \text{mod}\  \Im(\mu(0))$$
Therefore, the image of $W$ under $\delta_1$ is equal to the image of the composition
$$\xymatrix{H^0(L_1)\otimes H^0(L_{-1})\ar[r]^-{\kappa_1}&H^0(L^2_0)\ar[r]&\Coker(\mu(0)).}$$

$\delta_1$ in general is not injective. Let $\alpha$ be any section of $\mathcal{L}(-2X)$, $\beta$ be any section of $\mathcal{L}(-2Y)$. Then $(\alpha s_X^2)|_{C_0}\cdot(\beta s_Y^2)|_{C_0}\in W\subset \Ker(\mu(0))$, and $\delta_1((\alpha s_X^2)|_{C_0}\cdot(\beta s_Y^2)|_{C_0})=\frac{\alpha s_X^2\cdot\beta s_Y^2}{t}|_{C_0}\equiv0$. Clearly $(\alpha s_X^2)\cdot(\beta s_Y^2)\in \Ker(\mu_1)$ as in diagram (\ref{obstruction}), and therefore 
$$\delta_2((\alpha s_X^2)|_{C_0}\cdot(\beta s_Y^2)|_{C_0})=\frac{\alpha s_X^2\cdot\beta s_Y^2}{t^2}|_{C_0}=\alpha\beta|_{C_0}.$$ 
Thus the image of $\delta_2$ contains the image of the composition
$$\xymatrix{H^0(L_2)\otimes H^0(L_{-2})\ar[r]^-{\kappa_2}& H^0(L^2_0)\ar[r]&\Coker(\delta_1).}$$
Similarly, the image of $\delta_n$ contains the image of the composition
$$\xymatrix{H^0(L_n)\otimes H^0(L_{-n})\ar[r]^-{\kappa_n}&H^0(L^2_0)\ar[r]&\Coker(\delta_{n-1}).}$$
Therefore, we have a surjection

$$\xymatrix{\frac{H^0(L_0^2)}{\sum_{n=0}^a\Im(\kappa_n)}\ar@{->>}[r]&\Coker(\delta_a).}$$

The above analysis immediately gives a proof of theorem \ref{new} because if 
$$\xymatrix{\bigoplus_{n=0}^{a}H^0(C_0,L_n)\otimes H^0(C_0,L_{-n})\ar[r]^-{\kappa=\oplus_{n}\kappa_n}&H^0(C_0,L^2_0)}$$
is surjective (resp. of rank $=dim_{\mathbb{C}}\Sym^2H^0(L_0)$) for some $a\in\mathbb{Z}_+$, then $\delta_a$ is of maximal rank and therefore by proposition \ref{sufficient}, the multiplication map $\mu^2(t)$ is surjective (resp. injective) for small $t\ne0$. 

\section {proof of the main theorem}

We will prove theorem \ref{main thm} in this section.

Notice that in our degree and genus range, the MRC for $L$ is equivalent to the statement that $L$ is projectively normal. By theorem $(4.e.1)$ of \cite{Gr}, in our degree range, $H^0(L)\otimes H^0(L^k)\rightarrow H^0(L^{k+1})$ is surjective for any $k\ge2$. Thus to show such $L$ is projectively normal, it suffices to show the multiplication map $\mu^2$ in (\ref{multiplication}) is surjective. 

The idea is to make a good choice of $(C_0,L_0)$ such that the hypothesis of theorem \ref{new} is satisfied. We need the following lemma:
\begin{lemma}\label{tech}
Let $(C_0,L_0)$ be the same as theorem \ref{new}. Write $L_X=L_0|_X$ and $L_Y=L_0|_Y$. Suppose the restriction maps $H^0(C_0,L_n)\rightarrow H^0(X,L_X(np))$ and $H^0(C_0,L_n)\rightarrow H^0(Y,L_Y(-np))$ are surjective for $-a\le n\le a$. If 
\begin{eqnarray}\label{4.1}
\xymatrix{\bigoplus^a_{n=0}H^0(L_X(np))\otimes H^0(L_X(-np))\ar@{->>}[r]&H^0(L_X^2)}
\end{eqnarray}
and
\begin{eqnarray}\label{4.2}
\xymatrix{\bigoplus^a_{n=0}H^0(L_Y(np))\otimes H^0(L_Y((-n-1)p))\ar@{->>}[r]&H^0(L_Y^2(-p))}
\end{eqnarray}
are both surjective, then the natural map $\kappa$ in (\ref{kappa}) is surjective.

\end{lemma}
\begin{proof}
Let $s$ be any element of $H^0(C_0,L^2_0)$. Since (\ref{4.1}) is surjective, and any section of $L_X(np)$ extends to a section of $L_n$, we can modify $s$ by some element in the image of $\kappa$ such that $s|_X\equiv0$. Thus we can assume $s|_X\equiv0$, then $s|_Y\in H^0(L_Y^2(-p))$. Since (\ref{4.2}) is surjective, $s|_Y=\sum_{n=0}^a (x_ny_n)$ with $x_n\in H^0(L_Y(np))$, $y_n\in H^0(L_Y((-n-1)p))$. We can view $y_n$ as a section of $H^0(L_Y(-np))$ which vanishes at $p$, thus can be extended constantly $0$ to $X$ as a section of $L_n$. Still call it $y_n$. Extend $x_n$ arbitrarily to $X$ as a section of $L_{-n}$. Then $\sum_{n=0}^a\kappa(x_n\otimes y_n)=s$. 

\end{proof}

We now take $C_0=X\cup Y$, where $X$ and $Y$ are general curves of genus $g_X$, $g_Y$ meeting transversely at a general point $p$. In particular, $p$ is not a Weierstrass point of either $X$ or $Y$. We will divide the proof of the main theorem \ref{main thm} into two parts, according to the value of $h^1(L)$.

\subsection{\mbox{\boldmath$h^1(L)=2$} case} Since the residual series $N$ of $L$ is either a $g^1_{l+1}$ or $g^1_{l+2}$ depending on $g=2l$ even or $g=2l+1$ odd. We will work backwards by starting with a line bundle $N$ with $h^0(N)=1$ and take its residual. Here we take a simple $N$ whose restriction to $X$ and $Y$ are just suitable multiples of $\mathcal{O}_X(p)$ and $\mathcal{O}_Y(p)$, then take $L_0$ as the residual series of $N$. 

There are two subcases:
\begin{enumerate}
\item[(1)]  $g=2l$ even. Here we are dealing with $g^{l-1}_{3l-3}$'s. ($l\ge6$. $l=5$ needs a special argument and is proved in the appendix.) Let $g_X=g_Y=l$, $L_X=K_X(-\lceil\frac{l-1}{2}\rceil p)$ and $L_Y=K_Y(-\lfloor\frac{l-1}{2}\rfloor p)$.
\vspace{.3cm}
\item[(2)]  $g=2l+1$ odd. $L$ is a $g^{l-1}_{3l-2}$ ($l\ge6$). Take $g_X=l+1$, $g_Y=l$, $L_X=K_X(-\lceil\frac{l}{2}\rceil p)$, and $L_Y=K_Y(-\lfloor\frac{l}{2}\rfloor p)$
\vspace{.1cm}

\end{enumerate}

In both cases, it is easy to prove using the theory of limit linear series (see \cite{EH1}) that $(C_0, L_0)$ are smoothable in such a way all sections of $L_0$ extend to nearby. More precisely, the corresponding limit linear series on $C_0$ has aspects $V_X=(l-1)p+|K_X|$, $V_Y=(l-1)p+|K_Y|$ in case $g=2l$, and $V_X=(l-1)p+|K(-p)|$, $V_Y=lp+|K_Y|$ in case $g=2l+1$. They are both smoothable because the variety of limit liner series with the same ramification sequence as $(V_X,V_Y)$ at $p$ has expected dimensions.

\begin{prop} \label{hard}
For $(C_0,L_0)$ as described above, the natural map (\ref{kappa}) is surjective.
\end{prop}

\begin{proof}

Case (1). Here $L_n|_X=K_X((-\lceil\frac{l-1}{2}\rceil+n)p)$, $L_X^2=K_X^2(-2\lceil\frac{l-1}{2}\rceil p)$. Apply lemma \ref{lemma5.1} or \ref{lemma5.2} for $a=-\lfloor\frac{l-1}{2}\rfloor$, $L_X$ satisfies (\ref{4.1}). Meantime, $L_Y^2(-p)=K_Y^2(-(2\lfloor\frac{l-1}{2}\rfloor+1) p)$ and $2\lfloor\frac{l-1}{2}\rfloor+1$ is either $l-1$ or $l$ depending on $l$ even or odd. Again by lemma \ref{lemma5.1} or \ref{lemma5.2}, $L_Y$ satisfies (\ref{4.2}). 

Case (2). Take $a=\lceil\frac{l}{2}\rceil-1$. $L_n|_X=K_X((-\lceil\frac{l}{2}\rceil+n)p)$, $L_X^2=K_X^2(-2\lceil\frac{l}{2}\rceil p)$, $L_Y^2=K_Y^2(-2\lfloor\frac{l}{2}\rfloor p)$.

If $l$ even, by lemma \ref{lemma5.2}, we see that $L_Y$ satisfies (\ref{4.1}) and $L_X$ satisfies (\ref{4.2}) (notice $g_X=l+1$).

If $l$ odd, again by lemma \ref{lemma5.2}, $L_X$ satisfies (\ref{4.1}) and $L_Y$ satisfies (\ref{4.2}).

Thus in either case, the hypotheses of lemma \ref{tech} are satisfied, and therefore $\kappa$ in (\ref{kappa}) is surjective.
\end{proof}

\subsection{\mbox{\boldmath$h^1(L)=1$} case}

We are dealing with the residual series of $g^0_{l-1}$'s if $g=2l$ and $g^0_{l}$'s if $g=2l+1$, so smoothability is not a problem. Again there are two subcases:
\begin{enumerate}
\item[(1)] $g=2l$ even. $L$ is a $g^l_{3l-1}$. Let $g_X=g_Y=l$, $D_X$ (resp. $D_Y$) be a divisor consisting of $\lceil\frac{l-1}{2}\rceil$ (resp. $\lfloor\frac{l-1}{2}\rfloor$) general points on $X$ (resp. $Y$).
Take $L_X=K_X(p-D_X)$, $L_Y=K_Y(p-D_Y)$.
\vspace{.3cm}
\item[(2)] $g=2l+1$ odd. $L$ is a $g^l_{3l}$. Let $g_X=l+1$, $g_Y=l$ $D_X$ a general divisor of degree $\lceil\frac{l}{2}\rceil$ on $X$ and $D_Y$ a general divisor of degree $\lfloor\frac{l}{2}\rfloor$ on $Y$. Take $L_X=K_X(p-D_X)$, $L_Y=K_Y(p-D_Y)$.

\end{enumerate}

\begin{prop} \label{harder}For $(C_0,L_0)$ as described above, the natural map (\ref{kappa}) is surjective.

\end{prop}

\begin{proof}

Case (1). Let $M_X=L_X((\lceil\frac{l-1}{2}\rceil+1) p)=K_X((\lceil\frac{l-1}{2}\rceil+2) p-D_X)$, $M_Y=L_Y((\lfloor\frac{l-1}{2}\rfloor+1) p)=K_Y((\lfloor\frac{l-1}{2}\rfloor+2) p-D_Y)$.

If $l$ even, $L_X^2=K_X^2(2p-2D_X)=M^2_X(-(l+2)p)$. By lemma \ref{lemma5.3}, $L_X$ satisfies (\ref{4.1}). $L_Y^2(-p)=K_Y^2(p-2D_Y)=M^2_Y(-(l+1)p)$, $L_Y$ satisfies (\ref{4.2}). 

If $l$ odd, $L_X^2=K_X^2(2p-2D_X)=M_X^2(-(l+1)p)$, $L_X$ satisfies (\ref{4.1}). $L_Y^2(-p)=K^2_Y(p-D_Y)=M^2_Y(-(l+2)p)$, thus $L_Y$ satisfies (\ref{4.2}) by lemma \ref{lemma5.3}.
\vspace{.3cm}

Case (2). Let $M_X=L_X((\lceil\frac{l}{2}\rceil+1)p)=K_X((\lceil\frac{l}{2}\rceil+2)p-D_X)$, $M_Y=L_Y((\lfloor\frac{l}{2}\rfloor+1)p)=K_Y((\lfloor\frac{l}{2}\rfloor+2)p-D_Y)$.

If $l$ even, $L_Y^2=K^2_Y(2p-2D_Y)=M_Y^2(-(l+2)p)$, thus $L_Y$ satisfies (\ref{4.1}). $L_X^2(-p)=K_X^2(p-2D_X)=M_X^2(-(l+3)p)=M_X^2(-(g_X+2)p)$, $L_X$ satisfies (\ref{4.2}).

If $l$ odd, $L_X^2=K^2_X(2p-2D_X)=M_X^2(-(l+3)p)=M_X^2(-(g_X+2)p)$, $L_X$ satisfies (\ref{4.1}).
$L_Y^2(-p)=K_Y^2(p-2D_Y)=M_Y^2(-(l+2)p))$, $L_Y$ satisfies (\ref{4.2}).

Thus in either case, the hypotheses of lemma \ref{tech} are satisfied, and therefore $\kappa$ in (\ref{kappa}) is surjective.

\end{proof}
 
Combining proposition \ref{hard}, \ref{harder} and theorem \ref{new}, 
$$\xymatrix{\Sym^2H^0(C_t,L_t)\ar[r]^-{\mu^2(t)}&H^0(C_t,L_t^2)}$$
is surjective for small $t\ne0$.
Since there is a unique component $P$ of $\mathcal{W}^r_d$ which dominate $\mathcal{M}_g$ ($\rho>0$ case follows from the Gieseker-Petri theorem and the connectedness of $W^r_d(C)$; $\rho=0$ case follows from \cite{EH3}), to prove theorem \ref{main thm}, it suffices to arrange so that $(C_t,L_t)\in P$. But this is immediate because $C_0$ is a general point of the boundary of $\overline{\mathcal{M}_g}$.

\section{ some facts about canonical bundles on general curves}

 We present some technical facts about canonical bundles for a general curve in this section. They are needed in the proof of the main theorem. Lemma \ref{lemma5.1} to \ref{lemma5.3} are used in the previous section to show that maps of the type of (\ref{4.1}) and (\ref{4.2}) are surjective for the specific $L_0$ we choose in section $4$.
  
\begin{lemma} \label{lemma5.1}
 For a general smooth curve $X$ of genus $l\ge4$, and $p\in X$ a general point (in particular,not a Weierstrass point), the natural map
$$\xymatrix{\sum_{i+j=l-1,i,j\ge0}H^0(K_X(-ip))\otimes H^0(K_X(-jp))\ar[r]^-{m_{l-1}}& H^0(K_X^2(-(l-1)p))}$$
is surjective and
$$\xymatrix{\sum_{i+j=l,i,j\ge1}H^0(K_X(-ip))\otimes H^0(K_X(-jp))\ar[r]^-{m_l}& H^0(K_X^2(-lp))}$$ is of corank at most 1.

\end{lemma}

\begin{proof} Choose $\{\omega_0,...,\omega_{l-1}\}$ a basis of $H^0(K_X)$ adapted to the flag $H^0(K_X)\supsetneq H^0(K_X(-p))\supsetneq...\supsetneq H^0(K_X(-(l-1)p)$, i.e. $H^0(K_X(-ip))=span\{\omega_i,...,\omega_{l-1}\}$ for any $i$. By generality of $X$ and $p$, we can assume $K_X(-(l-2)p)$ is a base point free pencil. By base point free pencil trick, the kernel of the map
\begin{eqnarray}\label{m'}
\xymatrix{H^0(K_X)\otimes H^0(K_X(-(l-2)p)\ar[r]^-{m'}&H^0(K_X^2(-(l-2)p))}
\end{eqnarray}
is $H^0(K_X\otimes K_X^{-1}((l-2)p)=H^0(\mathcal{O}_X(l-2)p)$. By Riemann-Roch, $h^0(\mathcal{O}_X(l-2)p)=h^0(K_X(-(l-2)p))+l-2-l+1=1$. Thus, by dimension count, $m'$ is surjective with a one dimension kernel generated by $\omega_{l-1}\otimes\omega_{l-2}-\omega_{l-2}\otimes\omega_{l-1}$. We obtain a basis of $H^0(K_X^2(-(l-2)p)$:
\begin{eqnarray}\label{basis}
\xymatrix{\omega_{l-1}^2,\ \omega_{l-1}\omega_{l-2},\ ...,\ \omega_{l-1}\omega_1,\ \omega_{l-1}\omega_0,\\
\omega_{l-2}^2,\ \omega_{l-2}\omega_{l-3}\ ,...,\ \omega_{l-2}\omega_1,\ \omega_{l-2}\omega_0.}
\end{eqnarray}
Except $\omega_{l-2}\omega_0$, every other element of the above basis lies in the image of $m_{l-1}$, thus $m_{l-1}$ is surjective. Similarly, except $\omega_{l-2}\omega_0$, $\omega_{l-1}\omega_0$, $\omega_{l-2}\omega_1$, every other element of the above basis lies in the image of $m_l$. Therefore, $m_l$ is of corank at most $1$. 

\end{proof}

\begin{lemma}\label{lemma5.2}
For a general curve $X$ of genus $l\ge6$, and $p\in X$ a general point, $m_l$ in lemma \ref{lemma5.1} is surjective.
\end{lemma}

\begin{proof}
It suffices to find some $(X,p)$ for which $m_l$ is surjective. Take a special $(X,p)$ such that $K_X(-(l-2)p)$ is a pencil with a base point $q\ne p$ and that $q$ is not a base point of $K_X(-(l-3)p)$. This is equivalent to find a $X$ and $p$, $q\in X$ such that $h^0(\mathcal{O}_X((l-2)p+q))=2$ and $h^0(\mathcal{O}_X((l-3)p+q)=h^0(\mathcal{O}_X((l-2)p))=1$. One can actually choose $X$  to be a general point of $\mathcal{M}_l$ (but $(X,p)$ is not general in $\mathcal{M}_{l,1}$). This is because by theorem \ref{movex}, there exists a $g^1_{l-1}$ on a general curves $X$ with vanishing sequence $(0,l-2)$ at some point $p\in X$, and by theorem \ref{movnonex}, such a $g^1_{l-1}$ is base point free, complete and its residual series has a unique base point $q$.  For such $(X,p)$, any element in 
$$\xymatrix{V=\Im(\sum_{i=1}^2H^0(K_X(-ip))\otimes H^0(K_X(-(l-i)p))\ar[r]^-{m_l}&H^0(K_X^2(-lp)))}$$
vanishes at $q$. Now, take $\omega_{l-3}\in H^0(K_X(-(l-3)p))$ which does not vanish at $q$, then $\omega^2_{l-3}$ is in the image of $m_l$ but does not lie in $V$. It remains to show for this special $(X,p)$, $V$ is still of codimension  $1$ in $H^0(K_X^2(-lp))$. Since $K_X(-(l-2)p)$ has a unique base point $q$, by base point free pencil trick, the kernel of $m'$ in (\ref{m'}) is isomorphic to $H^0(\mathcal{O}_X((l-2)p+q)))$, which is $2$ dimensional. By dimension count, $m'$ is corank $1$, and there is exactly $1$ linear relation among the generators in (\ref{basis}). Let $\tau\in H^0(\mathcal{O}_X((l-2)p+q))$ be a section viewed as a rational function having a pole of order exactly $l-2$ at $p$ and a pole of order $1$ at $q$, then the kernel of $m'$ is spanned by 
$$\tau\omega_{l-1}\otimes\omega_{l-2}-\tau\omega_{l-2}\otimes\omega_{l-1}$$ 
and 
$$\omega_{l-1}\otimes\omega_{l-2}-\omega_{l-2}\otimes\omega_{l-1}.$$ 
Where $\{\omega_{l-2}, \omega_{l-1}\}$ span $H^0(K_X(-(l-2)p))$. Since a general curve only has normal Weierstrass points, $h^0(K_X(-(l-1)p)=1$. We can assume $\omega_{l-2}$ vanishes to order exactly $l-2$ at $p$. Thus $\tau\omega_{l-2}\in H^0(K_X)$ does not vanish at $p$, and the linear relation between the generators in (\ref{basis}) will have non-zero coefficient in $\omega_{l-1}\omega_0$. Thus 
 $$\xymatrix{\omega_{l-1}^2,\ \omega_{l-1}\omega_{l-2},\ ...,\ \omega_{l-1}\omega_1,\\
\omega_{l-2}^2,\ \omega_{l-2}\omega_{l-3}\ ,...,\ \omega_{l-2}\omega_2.}$$
are still linearly independent, and therefore $V$ is still of codimension $1$ in $H^0(K_X^2(-lp))$.

 \end{proof}
 
 \begin{lemma}\label{lemma5.3}
 Let $X$ be a general curve of genus $l\ge6$, $D$ is a divisor of degree $0<d<l-2$ consisting of $d$ distinct general points, $p\in X$ is a general point. Let $M=K_X(-D+(d+2)p)$. Then the multiplication maps
 \begin{eqnarray}\label{easy}
 \xymatrix{\sum_{i+j=l+1,i,j\ge0}H^0(M(-ip))\otimes H^0(M(-jp))\ar[r]&H^0(M^2(-(l+1)p))}
 \end{eqnarray} 
 and
  \begin{eqnarray}\label{diff}
 \xymatrix{\sum_{i+j=l+2,i,j\ge1}H^0(M(-ip))\otimes H^0(M(-jp))\ar[r]&H^0(M^2(-(l+2)p))}
 \end{eqnarray}  
 are both surjective.
 \end{lemma}

\begin{proof}

The idea is similar to the previous two lemmata. We have $deg\ M=2l$, $h^0(M)=l+1$ and $M(-lp)=K_X(-D+(d+2-l)p)$ is a base point free pencil since $p$ is general. Consider
\begin{eqnarray}
\xymatrix{H^0(M(-lp))\otimes H^0(M)\ar[r]^-{m'}&H^0(M^2(-lp))}
\end{eqnarray}
By base point free pencil trick, $\Ker(m')$ is isomorphic to $H^0(\mathcal{O}_X(lp))$ and is $1$ dimensional since $p$ is not a Weierstrass point of $X$. Since $h^0(M^2(-lp))=2l+1$, by dimension count, $m'$ is surjective. Extend a basis
$\{\omega_l,\ \omega_{l+1}\}$ of $H^0(M(-lp))$ to a basis of $H^0(M)$:
\begin{eqnarray}
\{\ \omega_0,\ \omega_1,...,\omega_{d},\ \widehat{\omega_{d+1}},\ \omega_{d+2},...,\omega_l,\ \omega_{l+1}\ \}
\end{eqnarray}
with each $\omega_i$ vanish to order exactly $i$ at $p$. The gap $\widehat\omega_{d+1}$ occurs because $h^0(M(-(d+1)p)=h^0(K_X(-D+p))=h^0(M(-(d+2)p)$. By the same argument as lemma \ref{lemma5.1}, (\ref{easy}) is surjective and (\ref{diff}) is at most corank $1$.

To prove that (\ref{diff}) is actually surjective for general $(X,D,p)$, we specialize to some $(X,D,p)$ such that $M(-lp)$ is a pencil with a base point $q$ but $q$ is not a base point of $M(-(l-1)p)$. This is equivalent to $h^0(M(-lp-q))=h^0(M(-lp))=2$ or equivalently $h^0(\mathcal{O}_X(D+q+(l-d-2)p))=2$ and $h^0(\mathcal{O}_X(D+(l-d-2)p))=1$. We can even choose $(X,p)$ be a general point of $\mathcal{M}_{l,1}$. This is because by the existence half of theorem \ref{fixex} (or by \cite{O}), there exists a $g^1_{l-1}$ with ramification sequence $(0,l-d-2)$ at $p$, and by the second half of theorem \ref{fixex}, a general such $g^1_{l-1}$ is base point free, complete, and its residual series has a unique base point $q$. Then, as in lemma \ref{lemma5.2}, $\omega_{l-1}^2$ is in the
 image of (\ref{diff}) but not in 
  $$\xymatrix{V=\Im(\sum_{i=1}^2H^0(M(-ip))\otimes H^0(M(-(l+2-i)p))\ar[r]&H^0(M^2(-(l+2)p)))}$$
It remains to prove for this special $(X,D,p)$, $V$ is still of codimension $1$ in $H^0(M^2(-(l+2)p))$. The problem here is that due to the base point $q$ of $M(-lp)$, $\Ker(m')$ is isomorphic to $H^0(\mathcal{O}_X((lp+q))$, which is $2$ dimensional (because $p$ is not a Weierstrass point). Thus $m'$ is not surjective but corank $1$. However, if $H^0(\mathcal{O}_X(ip+q))$ is span by $\{1,\tau\}$ with $\tau$ a rational function having pole of order exactly $l$ at $p$, then $\Ker(m')$ is spanned by
 $$\tau\omega_{l}\otimes\omega_{l+1}-\tau\omega_{l+1}\otimes\omega_{l}$$
  and 
 $$\omega_{l}\otimes\omega_{l+1}-\omega_{l+1}\otimes\omega_{l}.$$
   Again by theorem \ref{fixex}, we can even assume $h^0(M(-(l+1)p))=1$, and therefore can assume $m_l$ vanishes to order exactly $l$ at $p$. Thus $\tau\omega_l$ does not vanish at $p$ and by the same argument as lemma \ref{lemma5.2}, $V$ is still of codimension $1$ in $H^0(M^2(-(l+2)p))$.

\end{proof}

For the convenience of the reader, we state here some existence and non-existence results in Brill-Noether theory which we used in the proof of Lemma \ref{lemma5.2} and \ref{lemma5.3}. 

\begin{theorem}\label{fixex}(Brill-Noether theorem, fixed ramification point)
Let $X$ be a general curve of genus $g$, $p\in X$ a fixed general point. For any sequence $0\le m_0<...<m_r\le d$, let $\rho$ be the adjusted Brill-Noether number
$$\rho=g-\sum_{i=0}^r(m_i-i+g-d+r).$$
and let $\rho_+$ be the existence number
$$\rho_+=g-\sum_{m_i-i+g-d+r\ge0}(m_i-i+g-d+r)$$

If $\rho_+$ is nonnegative, then $X$ possesses $g^r_d$'s with vanishing sequence $(m_0,..,m_r)$ at $p$. Moreover, the variety $\mathcal{G}^r_d(m_0,...,m_r)$ parametrizing such $g^r_d$'s is empty if $\rho<0$ and has pure dimension $\rho$ if $\rho>0$. 
\end{theorem}

\begin{proof}
For the existence half, see \cite{L} theorem $3.2$-$1$. For the second half, see \cite{HM} theorem $5.37$.
\end{proof}

\begin{theorem}\label{movnonex}(Non-existence for low $\rho_{mov}$)

Let $X$ be a general curve of genus $g$, $\rho_{mov}$ be the moving-point Brill-Noether number
$$ \rho_{mov}=1+g-(r+1)(g-d+r)-\sum_{i=0}^r(m_i-i).$$
If $\rho_{mov}<1-r$, then for any $p\in X$, there is no $g^r_d$ on $X$ with vanishing sequence $(m_0,...,m_r)$ at $p$.

\end{theorem}
\begin{proof}
See \cite{L} theorem $4.3$-$6$.
\end{proof}

\begin{theorem}\label{movex}(Existence of $g^1_d$'s with movable ramification point)

Let $X$ be a general curve of genus $g$, fix $0<m_1\le d$, then there exists a $g^1_d$ on $X$ with vanishing sequence $(0,m_1)$ at some point $p\in X$ if $g-d+1\ge0$ and the moving-point Brill-Noether number $\rho_{mov}=1+g-2(g-d+1)-(m_1-1)\ge0$.
\end{theorem}

\begin{proof}
See \cite{L} theorem $3.3$-$4$ and example $3.3$-$8$.
\end{proof}

\section{appendix}

We give a special treatment for the case of genus $10$ curves $C$ and a complete $g^4_{12}$ $|L|$ on $C$. We include this case here because it does not follow from the general discuss in section 4 (we need $l\ge6$ in lemma \ref{lemma5.2}) and there is very interesting geometry behind this example.

Since $\rho=g-(r+1)(g-d+r)=0$ in this case, by Brill-Noether Theorem, a general genus $10$ curve $C$ has only finitely many $g^4_{12}$s, and each  $g^4_{12}$ on $C$ is very ample. It is also known that for general such $C$, there exists some $g^4_{12}$ $|L|$ such that $\mu^2$ is not injective, i.e. $C$ is contained in some quadric hypersurface in $\mathbb{P}^4$ under the embedding of $|L|$, if and only if $C$ is contained in some $K3$ surface (see \cite{FP}). By proposition $2.2$ of \cite{CU}, the locus $\mathcal{K}$ in $\mathcal{M}_{10}$ consisting of curves contained in some $K3$ surface is a divisor. (Interestingly, Farkas and Popa \cite{FP} proved that $\mathcal{K}$ is a counter example for the slope conjecture.) Thus, for a general genus $10$ curve $C$,
and any $g^4_{12}$ $|L|$ on $C$, the multiplication map $\mu^2$ is injective (therefore an isomorphism, since the domain and range of $\mu^2$ are of the same dimension). 

Here we will use the results in the previous section to give a proof of this fact without using the geometry of curves contained in $K3$ surfaces.

\begin{proof}
    Notation the same as theorem \ref{new}. Take $X$ and $Y$ both general curves of genus $5$ meeting at a general point $p$. Take $L_X=K_X(-2p)=g^2_6$, $L_Y=K_Y(-2p)$ (which is smoothable).
    We will check $\kappa$ in (\ref{kappa}) is surjective.
    Consider the following exact sequence of sheaves:
    $$\xymatrix{0\ar[r]&\mathcal{O}_{C_0'}(2)\ar[r]&\pi_*\mathcal{O}_{C_0}(2)\ar[r]&\bigoplus_{k=1}^5(\mathbb{C}_{p_k}\oplus \mathbb{C}_{q_k})\ar[r]&0}.$$
    Where $\pi: C_0\longrightarrow C_0' \subset\mathbb{P}^4$ is the map given by the linear series $|L_0|$. $C_0'=X'\cup Y'$ is the image of $\pi$ consisting of two degree $6$ plane curves with $5$ nodes, and $\bigoplus_{k=1}^5(\mathbb{C}_{p_k}\oplus \mathbb{C}_{q_k})$ is the skyscraper sheaf of rank $10$ supported on the five nodes of $X'$ and five nodes on $Y'$.
    
    Taking the long exact cohomology sequence, we obtain
    
    $$\xymatrix{0\ar[r]&H^0(\mathcal{O}_{C_0'}(2))\ar[r]&H^0(\mathcal{O}_{C_0}(2))\ar[r]^{\phi\ \ \ \ }&\bigoplus_{k=1}^5(\mathbb{C}_{p_k}\oplus \mathbb{C}_{q_k})\ar[r]&H^1(\mathcal{O}_{C_0'}(2))\ar[r]&0}$$
    
    From the above exact sequence we see that a section $s\in H^0(L^2_0)$ is not coming from pull back of $H^0(\mathcal{O}_{\mathbb{P}^4}(2))$ if and only if $\phi(s)$ is not zero in $\mathbb{C}^{10}$, i.e. $s$ separate at least one node on $C_0'$.
    
    Now, choose $\sigma_i$ sections of $L_{1}$, $\sigma_j$ sections of $L_{-1}$ according to their value at the inverse image under $\pi$ of the ten nodes on $C_0'$ as table 1 below. Where $i=0,1$, $j=2,3$, $p_k'$, $p_k''$ (resp. $q_k'$, $q_k''$) are points in $X$ (resp. $Y$) which gets mapped to the node $p_k$ (resp. $q_k$) on $X'$ (resp. $Y'$). Here it suffices to consider two nodes on each component, say $k=1,2$. Cross means that $\sigma_i$ does not vanish at the corresponding point, 0 means vanishing. For instance, $\sigma_0$ is a section of $L_1$, such that $\sigma_0|_X$ is a section of the $g^3_7=K_X(-p)$ on $X$ that vanish on $p_1''$ and $p_2''$ but not on $p_1'$ or $p_2'$ (Although the $g^2_6=K_X(-2p)$ does not separate $p_1'$, $p_1''$ or $p_2'$, $p_2''$, the $g^3_7$ does). $\sigma_0|_Y$ is a section of the $g^1_5=K_Y(-3p)$ on Y that vanishes on $q_1'$, $q_1''$, but not on $q_2'$ or $q_2''$. Similarly, for other $\sigma_i$. By the generality of $X$, $Y$ and $p$, the assigned value in table $1$ can be achieved. 
    
     \begin{table}[h]

\caption{}
\begin{tabular}{c|cc|cc|cc|cc}\hline
&$p_1'$&$p_1''$&$p_2'$&$p_2''$&$q_1'$&$q_1''$&$q_2'$&$q_2''$\\
\hline
$\sigma_0$&$\times$&$0$&$\times$&$0$&$\times$&$\times$&$0$&$0$\\
$\sigma_1$&$\times$&$0$&$\times$&$0$&$0$&$0$&$\times$&$\times$\\
$\sigma_2$&$\times$&$\times$&$0$&$0$&$\times$&$0$&$\times$&$0$\\
$\sigma_3$&$0$&$0$&$\times$&$\times$&$\times$&$0$&$\times$&$0$\\
\hline
   \end{tabular}
\end{table}
    
    Table $2$ describes the difference of $\sigma_i\sigma_j$ at $p_k'$, $p_k''$ and $q_k'$, $q_k''$ for $k=1,2$.
    \begin{table}[h]
    \caption{}
    \begin{tabular}{c|cccc}\hline
    &$p_1$&$p_2$&$q_1$&$q_2$\\\hline
    $\phi(\sigma_0\sigma_2)$&$\times$&$0$&$\times$&$0$\\
    $\phi(\sigma_0\sigma_3)$&$0$&$\times$&$\times$&$0$\\
    $\phi(\sigma_1\sigma_2)$&$\times$&$0$&$0$&$\times$\\
    $\phi(\sigma_1\sigma_3)$&$0$&$\times$&$0$&$\times$\\
    \hline
    \end{tabular}
    \end{table}

    From table 2, we get a matrix of rank at least 3. Thus, $\Im(\delta_1)$ is mapped under $\phi$ to a subspace of dimension at least $3$ in $\mathbb{C}^{10}$. In other words, we have shown that $C_t$ for $t\neq0$ is contained in at most one quadric in $\mathbb{P}^4$.
    
It remains to show that we get an extra dimension from
$$\Im(H^0(L_2)\otimes H^0(L_{-2})\longrightarrow H^0(L_0^2)).$$ 
    
 Choose $\lambda_1\in H^0(L_2)$, $\lambda_2\in H^0(L_{-2})$ according to table 3,
    
   \begin{table}[h]

\caption{}
\begin{tabular}{c|cc|cc|cc|cc}\hline
&$p_1'$&$p_1''$&$p_2'$&$p_2''$&$q_1'$&$q_1''$&$q_2'$&$q_2''$\\
\hline

$\lambda_1$&$\times$&$0$&$0$&$0$&$\times$&$\times$&$\times$&$\times$\\
$\lambda_2$&$\times$&$\times$&$\times$&$\times$&$\times$&$0$&$0$&$0$\\
\hline
   \end{tabular}
\end{table}
    
We get one more vector $\phi(\lambda_1\lambda_2)$ in $\mathbb{C}^{10}$. So we can add one row to the matrix in table 2, to get table 4

 \begin{table}[h]
    \caption{}
    \begin{tabular}{c|cccc}\hline
    &$p_1$&$p_2$&$q_1$&$q_2$\\\hline
    $\phi(\sigma_0\sigma_2)$&$\times$&$0$&$\times$&$0$\\
    $\phi(\sigma_0\sigma_3)$&$0$&$\times$&$\times$&$0$\\
    $\phi(\sigma_1\sigma_2)$&$\times$&$0$&$0$&$\times$\\
    $\phi(\sigma_1\sigma_3)$&$0$&$\times$&$0$&$\times$\\
    $\phi(\lambda_1\lambda_2)$&$\times$&$0$&$\times$&$0$\\
    \hline
    \end{tabular}
    \end{table}

To show the matrix in table $4$ has rank $4$, it suffices to show that the first row and the last row can be chosen linearly independently, or equivalently, that 
$$\frac{\sigma_0|_X\cdot\sigma_2|_X}{\lambda_1|_X\cdot\lambda_2|_X}(p_1')\neq\frac{\sigma_0|_Y\cdot\sigma_2|_Y}{\lambda_1|_Y\cdot\lambda_2|_Y}(q_1').$$

This can be easily achieved, for instance, as follows. Take $X=Y$ and $X$, $Y$ meeting at the same point $p\in X=Y$ and $q_1=p_3$, $q_2=p_2$. Choose $\sigma_0|_X=\sigma_2|_Y$, $\sigma_0|_Y=\sigma_2|_X$ $\lambda_1|_X=\lambda_2|_Y$ and $\lambda_1|_Y=\lambda_2|_X$ as the unique (up to scalar) sections satisfying the conditions in table $5$:
 
\begin{table}[h]
    \caption{}
    \begin{tabular}{c|cc|cc|cc}\hline
    &$p_1'$&$p_1''$&$p_2'$&$p_2''$&$p_3'$&$p_3''$\\\hline
    $\sigma_0|_X=\sigma_2|_Y$&$\times$&$0$&$\times$&$0$&$\times$&$0$\\
    $\sigma_0|_Y=\sigma_2|_X$&$\times$&$\times$&$0$&$0$&$\times$&$\times$\\
    $\lambda_1|_X=\lambda_2|_Y$&$\times$&$0$&$0$&$0$&$\times$&$0$\\
    $\lambda_1|_Y=\lambda_2|_X$&$\times$&$\times$&$\times$&$\times$&$\times$&$\times$\\
    \end{tabular}
    \end{table}
    
    Then 
    $$\frac{\sigma_0|_X\cdot\sigma_2|_X}{\lambda_1|_X\cdot\lambda_2|_X}=\frac{\sigma_0|_Y\cdot\sigma_2|_Y}{\lambda_1|_Y\cdot\lambda_2|_Y}$$
    as rational functions, but since everything is general, 
    $$\frac{\sigma_0|_X\cdot\sigma_2|_X}{\lambda_1|_X\cdot\lambda_2|_X}(p_1')\neq\frac{\sigma_0|_Y\cdot\sigma_2|_Y}{\lambda_1|_Y\cdot\lambda_2|_Y}(p_3').$$
    
    In conclusion, we can arrange so that the rank of the matrix in table $4$ is exactly $4$ and therefore the image in $\mathbb{P}^4$ of a general genus $10$ curve under a general (thus every) $g^4_{12}$ is not contained in any quadric. 

\end{proof}

\begin{corollary} \label{cor}
 For $g>10$, a general curve in $\mathbb{P}^4$ with degree $d\ge\frac{4}{5} g+4$ is not contained in any quadric. 

\end{corollary}

\begin{proof}
Consider the curve consists of a general curve $X$ of genus $10$ and a general curve $Y$ of genus $g-10$ meeting at a general point $p$. Consider the limit linear series with aspects $V_X=(d-12)p+|g^4_{12}|$, $V_Y=8p+|g^4_{d-8}|$. Since everything is general, $V_X$ has vanishing sequence $(8,9,10,11,12)$ at $p$, and $V_Y$ has vanishing sequence $(d-12,d-11,...,d-8)$. Since limit $g^4_d$ on $X\cup Y$ with the above specified vanishing sequence have dimension  $\rho(g-10, 4, d-8)+\rho(10, 4, 12)=(g-10)-5(g-10-(d-8)+4)+0=g-5(g-d+4)=\rho(g, 4, d)$, by the smoothing theorem of limit linear series (see \cite{EH1}), $V_X$, $V_Y$  are smoothable. On the other hand, the image of $X$ in $\mathbb{P}^4$ under $\phi_{V_X}$ is not contained in any quadric and $V_X|_Y$ is a $|g^0_{d-12}|$. Thus, by degenerating to such limit $g^4_{12}$ on $X\cup Y$, we have our conclusion.

\end{proof}


\begin{thebibliography}{999}

\bibitem{AC} Arbarello, E, Cornalba, M.: Su una congettura di Petri. Comment. Math. Helvetici {\bf56}, 1-38, 1981.

\bibitem{ACGH} Arbarello, E., Cornalba, M., Griffiths, P., Harris, J.: Geometry of Algebraic Curves, Volume  I. Springer Grundlehren {\bf267}, 1985. 

\bibitem{ACGH1} Arbarello, E., Cornalba, M., Griffiths, P., Harris, J.: Special Divisors on Algebraic Curves. Regional Algebraic Geometry Conference. Athens, Georgia, May, 1979.

\bibitem{BE} Ballico, E., Ellia, P.:The maximal rank conjecture for nonspecial curves in $\mathbb{P}^3$. Invent. Math. {\bf79}, 541-555, 1985.

\bibitem{BE1} Ballico, E., Ellia, P: On postulation of curves in $\mathbb{P}^4$. Math. Z. {\bf188}, 215-223, 1985.

\bibitem{BE2} Ballico, E., Ellia, P.: The maximal rank conjecture for nonspecial curves in $\mathbb{P}^n$. Math. Z. {\bf196}, 355-367, 1987.

\bibitem{C} Clemens, H.: A local proof of Petri's conjecture at the general curve. J. Differential Geom. {\bf54}, 139-176,  2000.

\bibitem{CU} Cukierman, F., Ulmer, D.: Curves of genus ten on $K3$ surfaces. Compositio Mathematica {\bf89}, 81-90, 1993.

\bibitem{EH} Eisenbud, D., Harris, J.:  A simpler proof of the Gieseker-Petri theorem on special divisors. Invent. Math. {\bf74}, 269-280, 1983.

\bibitem{EH2} Eisenbud, D., Harris, J.: Divisors on general curves and cuspidal rational curves. Invent. Math. {\bf74}, 371-418, 1983.

\bibitem{EH1} Eisenbud, D., Harris, J.: Limit linear series: Basic theory. Invent. Math. {\bf85}, 337-371, 1986.

\bibitem{EH3} Eisenbud, D., Harris, J.: Irreducibility and monodromy of some families of linear series. Annales scientifiques de l'\'{E}cole Normale Sup\'{e}rieure, S\'{e}r.4, {\bf20}, 65-87, 1987.

\bibitem{FP} Farkas, G., Popa, M.:  Effective divisors on $\mathcal{M}_g$, curves on K3 surfaces, and the slope conjecture.  J. Algebraic Geom. {\bf14(2)}, 241--267, 2005.

\bibitem{G} Gieseker, D.: Stable curves and special divisors. Invent. Math. {\bf66}, 251-275, 1982.

\bibitem{Gr} Green, M., Koszul cohomology and the geometry of projective varieties. J. Differential Geom. {\bf19} 125-171, 1984.

\bibitem{GH} Griffiths, P., Harris, J.: On the variety of linear systems on a general algebraic curve.  Duke Math. J. {\bf47}, 233-272, 1980.

\bibitem{GL1} Green, M., Lazarsfeld, R.: On the projective normality of complete linear series on an algebraic curve. Invent. Math. {\bf83}, 73-90, 1986. 

\bibitem{GL} Green, M., Lazarsfeld, R.: Deformation theory, generic vanishing theorems, and some conjectures of Enriques, Catanese and Beauville. Invent. Math. {\bf90}, 389-407, 1987.

\bibitem{H} Harris, J.:  Curves in projective space. Les Press de l'Universit$\acute{e}$ de Montr$\acute{e}$al, 1982. 
 
\bibitem{HM} Harris, J., Morrison, I.:  Moduli of Curves. Graduate Text in Mathematics {\bf187}, Springer-Verlag New York, 1998.

\bibitem{L} Lehman, R.: Brill-Noether type theorems with a movable ramification point. Ph.D thesis, MIT, 2007. http://dspace.mit.edu/handle/1721.1/38885

\bibitem{O} Osserman, B.: The number of linear series on curves with given ramification. Int. Math. Res. Not. {\bf47}, 2513-2527, 2003.

\bibitem{Z} Ran, Z.: A note on Hilbert schemes of nodal curves.  Journal of Algebra, {\bf 292}, 429-446, 2005.

\bibitem{S} Sernesi, E.: Deformations of Algebraic Schemes. Springer Grundlehren {\bf334}, 2006.

\bibitem{W} Wang, J.: Deformations of pairs $(X,L)$ when $X$ is singular. 
http://arxiv.org/abs/1003.6073
\end{thebibliography}
\end{document}